\documentclass[a4paper, 12pt]{article}

\usepackage[utf8]{inputenc}
\usepackage[T1]{fontenc}
\usepackage[a4paper, margin=2.5cm]{geometry}
\usepackage{needspace}

\usepackage{libertine} % text font
\usepackage{inconsolata} % texttt font

\usepackage{amsmath, amsthm, amssymb}
\usepackage{graphicx}
\usepackage{enumerate}
\usepackage{authblk}
\usepackage[textsize=small, textwidth=2cm, color=yellow]{todonotes}
\usepackage{thmtools}
\usepackage{thm-restate}
\usepackage[colorlinks=true, citecolor=red]{hyperref}
\usepackage{float}
\usepackage{amsfonts,mathtools,enumerate}
\usepackage{comment}
\usepackage{caption}
\numberwithin{figure}{section}
\renewenvironment{abstract}
{\small\vspace{-1em}
\begin{center}
\bfseries\abstractname\vspace{-.5em}\vspace{0pt}
\end{center}
\list{}{
\setlength{\leftmargin}{0.6in}%
\setlength{\rightmargin}{\leftmargin}}%
\item\relax}
{\endlist}

\declaretheorem[name=Theorem]{theorem}
\declaretheorem[name=Lemma, sibling=theorem]{lemma}

\declaretheorem[name=Corollary, sibling=theorem]{corollary}
\declaretheorem[name=Conjecture, sibling=theorem]{conjecture}

\declaretheorem[name=Claim, sibling=theorem]{claim}

\def\cqedsymbol{\ifmmode$\lrcorner$\else{\unskip\nobreak\hfil
\penalty50\hskip1em\null\nobreak\hfil$\lrcorner$
\parfillskip=0pt\finalhyphendemerits=0\endgraf}\fi}

\DeclareMathOperator{\diam}{diam}
\DeclareMathOperator{\rad}{radius}

\DeclarePairedDelimiter{\ceil}{\lceil}{\rceil}

\def\C{\mathcal{C}} % class
\newcommand{\NN}{\mathbb{N}} % integers
\def\R{\mathcal{R}} % realize
\newcommand{\csqrt}[1]{\lceil \sqrt{#1}\rceil} % ceiling of sqrt

\let\le\leqslant
\let\ge\geqslant
\let\leq\leqslant
\let\geq\geqslant

\title{Improved Pyrotechnics:\\ Closer to the Burning Number Conjecture}

\author[1]{Paul Bastide}
\author[2]{Marthe Bonamy}
\author[3]{Anthony Bonato}
\author[4]{Pierre Charbit}
\author[5]{Shahin Kamali}
\author[6]{Théo Pierron}
\author[4]{Mikaël Rabie}
\affil[1]{\'Ecole Normale Sup\'erieure  Rennes, Rennes, France}
\affil[2]{CNRS, LaBRI, Universit\'e de Bordeaux, Bordeaux, France}
\affil[3]{Ryerson University, Toronto, Canada}
\affil[4]{IRIF, Paris, Paris, France}
\affil[5]{University of Manitoba, Winnipeg, Canada}
\affil[6]{LIRIS, Universit\'e de Lyon, Lyon, France}
\date{\today}

\begin{document}

%%localdefs
\def\fk{4k^2+5k-2}

\sloppy
\maketitle

\begin{abstract}
The Burning Number Conjecture claims that for every connected graph $G$ of order $n,$ its burning number satisfies $b(G) \le \lceil \sqrt{n} \rceil.$ While the conjecture remains open, we prove that it is asymptotically true when the order of the graph is much larger than its \emph{growth}, which is the maximal distance of a vertex to a well-chosen path in the graph. We prove that the conjecture for graphs of bounded growth reduces to a finite number of cases. We provide the best-known bound on the burning number of a connected graph $G$ of order $n,$ given by $b(G) \le \sqrt{4n/3} + 1,$ improving on the previously known $\sqrt{3n/2}+O(1)$ bound. Using the improved upper bound, we show that the conjecture almost holds for all graphs with minimum degree at least $3$ and holds for all large enough graphs with minimum degree at least $4$. The previous best-known result was for graphs with minimum degree $23$.

\end{abstract}

% edition
%\overfullrule=50pt % to spot overfull hbox
\section{Introduction}\label{sec:intro}

Graph burning is a simplified model for the spread of influence in a network. Associated with the process is a parameter introduced in \cite{BJR0,bonato_intro,thez}, the burning number, which quantifies the speed at which the influence spreads to every vertex. Graph burning is defined as follows. Given a graph $G$, the burning process on $G$ is a discrete-time process. %Vertices may be either unburned or burned throughout the process.
At the beginning of the first round, all vertices are unburned. In each round, first all unburned vertices that have a burned neighbour become burned, and then one new unburned vertex is chosen to burn, if such a vertex is available. If at the end of round $k$ every vertex of $G$ is burned, then $G$ is $k$-\emph{burnable}. The \emph{burning number} of $G,$ written $b(G),$ is defined to be the least $k$ such that $G$ is $k$-burnable.

As a graph with $n$ isolated vertices is not $(n-1)$-burnable, we instead focus on connected graphs. Paths are an interesting special case. For a path $P_n$ on $n$ vertices, it was shown in \cite{bonato_intro} that $b(P_n)=\lceil \sqrt{n} \rceil$. In \cite{bonato_intro}, it was conjectured that paths have the largest burning number among connected graphs.

\medskip

\noindent \textbf{Burning Number Conjecture} (or \textbf{BNC}):\index{Burning Number Conjecture} For a connected graph $G$ of order $n$, $$b(G)\leq \lceil \sqrt{n}\rceil.$$

\smallskip

For any connected graph $G$ and any spanning tree $T$ of $G$, we have $b(G) \leq b(T)$. Therefore it is sufficient to prove that the conjecture holds for trees.

The BNC has resisted attempts at its resolution, although various upper bounds on the burning number are known. In \cite{bessy2}, it was proved that for every connected graph $G$ of order $n$ that
$$b(G)\leq \sqrt{{\frac{12}7 n}} +3.$$ These bounds were improved in \cite{land_approx}, who proved, up until this paper, the best-known upper bound:
$$b(G)\le \bigg \lceil
\frac{\sqrt{24n+33}-3}{4} \bigg \rceil=\sqrt{\frac32n} +O(1).$$

In the present paper, we improve this by showing that any tree of order $n$ can burn in $\sqrt{\frac43 n}+1$ rounds. %This improves upon the $\sqrt{\frac32 n}+O(1)$ bound in \cite{land_approx}, and gives the best-known upper bound on the burning number of a connected graph.

\begin{theorem}
    \label{th:approx}
    For a connected graph $G$ of order $n,$ we have that
    $$
        b(G) \leq \left\lceil \sqrt{{\frac43 n}} \right\rceil + 1.
    $$
\end{theorem}

While the BNC is open for general graphs, it is known to hold for a number of graph classes. For example, in \cite{bonato_spider}, it is proven that the conjecture holds for spiders, which are defined as trees with at most one vertex of degree at most $3$. In \cite{kamali}, it was proven that any graph with minimum degree $\delta \ge 23$ satisfy the conjecture. %Although this result encompasses a large class of graphs, it omits the class of trees. 
Using Theorem \ref{th:approx}, we prove that the conjecture almost holds for all connected graphs of minimum degree at least $3$, and also holds for those of minimum degree at least $4$ that are of a large enough order.

\begin{theorem}\label{th:degree3}
For any connected graph $G$ on $n$ vertices, if every vertex in $G$ has degree at least $3$, then
$
    b(G) \leq \left\lceil \sqrt{n} \right\rceil+2.
$
\end{theorem}

\begin{theorem}\label{th:degree4}
For any connected graph $G$ on $n$ vertices, if every vertex in $G$ has degree at least $4$ and $n$ is sufficiently large, then
$
    b(G) \leq \left\lceil \sqrt{n} \right\rceil.
$
\end{theorem}

Our second main result concerns trees that have bounded growth. The \emph{growth} of a connected graph $G$ is the smallest integer $k$ such that all vertices in $G$ are within distance $k$ of some path $P$ in $G$. Trees of growth $1$ are known also as the class of \emph{caterpillars} and by extension trees of growth $k$ are often referred to as \emph{$k$-caterpillars}. As proven in Hiller, Triesch, and Koster~\cite{hiller_cat} and independently in \cite{liu1}, caterpillars satisfy the conjecture, and the conjecture has also been confirmed for $2$-caterpillars (or \emph{lobsters}) in \cite{hiller_cat}.  We refer the reader to a survey on graph burning for further details~\cite{bonato_survey}. Note that any result on trees of bounded growth immediately applies to graphs of bounded growth, by considering an appropriate spanning tree. 

In this paper, we present two theorems about trees of bounded growth.
The first one states that in order to prove the BNC for $k$-caterpillars, it suffices to verify it for a finite number of them. It would be interesting to find ways to go further and transform this into a practical way for the BNC to be verified for $k$-caterpillars (for say $k=10$). The statement currently suggests unreasonably many cases, which seems unnecessary.
\begin{theorem}\label{th:reduced}
For any $k \in \mathbb{N}$, if the BNC holds for all $k$-caterpillars on at most $(\fk)^2$ vertices, then it holds for all $k$-caterpillars.
\end{theorem}

The second gives an approximation of the BNC for all $k$-caterpillars.

\begin{theorem}\label{th:approxcat}
For any $k \in \mathbb{N}$ and every tree $T$ of growth at most $k$ on $n$ vertices
$$b(T)\leq  \sqrt{n+(\fk)^2}.$$
\end{theorem}

In fact both will be direct consequences of the more general statement below (by taking $c=0$ or $c=(\fk)^2$).

\begin{theorem}\label{th:maincat}
Let $k$ and $c$ be two integers. If $b(T)\leq \sqrt{|V(T)|+c}$ holds for every $k$-caterpillar $T$ of order at most $(\fk)^2-c$, then it holds for all $k$-caterpillars. \end{theorem}

Before closing this introduction, we present a reformulation of the conjecture that we will use in the rest of the paper, and allows us to propose a stronger conjecture. If we denote by $v_i$ the vertex chosen to burn at the end of round $i$ in a burning process, then the set of vertices that are burned after $p$ rounds is simply $\cup_{i=1,\ldots,p} B(v_i,p-i)$ where $B(v,r)$ is the set of the vertices at distance at most $r$ from $v$. Therefore, the statement of the BNC can be reformulated as the fact that the vertex set of any graph on $n$ vertices can be covered by balls of radius $0,1,\ldots,\csqrt n -1$. One can therefore wonder now which sets of radii are enough to cover any $n$ vertex graph.

A graph $G$ is \emph{$\R$-burnable} for some finite set of $\R = \{r_1,\ldots,r_p\} \subseteq \NN$ if there exist vertices $v_1,\ldots,v_p$ of $G$ such that $\bigcup_{i \in \{1,\ldots,p\}} B(v_i,r_i)  = V(G)$. Note that multisets are not eligible, which is crucial in later arguments. As referenced earlier, the BNC states that every graph on $n$ vertices is $\{0,1,\ldots,\csqrt n -1\}$-burnable. We propose the stronger following conjecture.
\begin{conjecture}\label{bnc:strong}
Let $T$ be a tree of growth $k$ and of order $n$. If $\R$ is a finite set of integers such that $\{0,1,\ldots,k\} \subseteq \R$ and $\sum_{r \in \R} (2r+1)\geq n$, then $T$ is $\R$-burnable.
\end{conjecture}

Observe that $$\sum_{r=0}^{ \csqrt n -1} (2r+1)=\csqrt n^2\geq n.$$

The article is organized as follows. In Section \ref{sec:gen}, we prove Theorem \ref{th:approx} and its two corollaries Theorem \ref{th:degree3} and Theorem \ref{th:degree4}. In Section \ref{sec:cat}, we prove Theorem \ref{th:maincat}. The concluding section describes further directions.

All graphs we consider are simple, finite, and undirected. The distance between vertices $u$ and $v$ is denoted by $d(u,v).$ The diameter of a graph $G$ is denoted by $\diam (G)$. For a vertex $v$ and an integer $r$, we denote by $B(v,r)$ the ball of radius $r$ centered on $v$; that is, the set of vertices at distance at most $r$ from $v$. We will use the notation $\rad(G)$ for the radius of a graph, that is the smallest integer $r$ such that there exists a vertex $x$ of $G$ with $V(G)=B(x,r)$. For background on graph theory, see \cite{west}.
%------------------------------------------------------------------------------------------------------------------------------%

\section{General bounds}\label{sec:gen}

\subsection{Improved upper bound on the burning number}

The goal of this section is to prove Theorem~\ref{th:approx}. The proof relies heavily on the following lemma. 

\begin{lemma}
    \label{lem:approx_ind}
    For every tree $T$ and nonempty and finite $\R\subseteq \NN$, there exists $r\in \R$ such that either $\rad(T)\leq r$ or there exists a subtree $T'$ of $T$ such that 
    \begin{enumerate}
    \item $T\setminus T'$ is connected,
    \item $\rad(T')\leq r$, and
    \item $|V(T')|\geq r+|\R|/2$.
    %\item %$r amongst the $|\R|/2$ largest elements of $\R$ (not useful for the theorem)
    \end{enumerate}
\end{lemma}

We first show how to quickly derive Theorem~\ref{th:approx} assuming that the lemma holds. 

\begin{proof}[Proof of Theorem~\ref{th:approx}]
Given a tree $T$ on $n$ vertices and an integer $p$, we repeatedly apply Lemma~\ref{lem:approx_ind} to obtain that $T$ is $[0,p]$-burnable unless $$n > \sum_{i=0}^{p}i+\sum_{i=0}^pi/2=\frac{3p(p+1)}{4}.$$ We therefore have that $p \leq \left\lceil \sqrt{{\frac43 n}} \right\rceil + 1,$ as desired.
\end{proof}

\begin{proof}[Proof of Lemma~\ref{lem:approx_ind}] Let $T$ be a tree and $\R\subseteq \NN$ satisfying its hypotheses. We consider a maximum length path $P=(t_0,t_1,\ldots,t_\ell)$ of $T$, and root $T$ in $t_\ell$. For $0\leq i \leq \ell$, we define $T_i$ to be the subtree of $T$ rooted in $t_i$, containing all descendants of $t_i$ including itself. %Note that for every $i$, $T\setminus T_i$ is connected.
For $0\leq i \leq \ell$, we also define:

    $$\phi(i) = \max\{j\colon T_j\subseteq B(t_i,i)\}$$
    
It is straightforward to notice that  $i\leq \phi(i)\leq 2i$. Further, if $\phi(r)=\ell$ for some $r$, then $T=T_{\ell}\subseteq B(t_r,r)$ so we have the first possible outcome of the lemma. We therefore assume from now on that $\phi(r)<\ell$ for every $r\in \R$% (and in fact every $r$ smaller than $max \R$)
. For any $r\in \R$, note that the tree $T'=T_{\phi(r)}$ satisfies the two first conditions required by the second outcome of the lemma, so we may also assume that for any $r\in \R$, $T_{\phi(r)}$ contains strictly less than $r+|\R|/2$ vertices. In particular, since vertices $t_0,t_1,\ldots, t_{\phi(r)}$ belong to $T_{\phi(r)}$, we have that $\phi(r)< r+|\R|/2-1$ for every $r\in \R$.

Let $\R'$ be the set of the $N=\lceil |\R|/2 \rceil$ largest elements of $\R$ and $r\in\R'$. A straightforward but crucial consequence of the fact that $\R$ is not a multiset is that $r\geq |\R|-N \geq |\R|/2-1 $ and hence $\phi(r)<2r$. Further, by the definition of $\phi(r)$ and since $\phi(r)<\ell$, there exists a vertex $x_r\in T_{\phi(r)+1}\setminus T_{\phi(r)}$ such that $d(t_r,x_r)=r+1$. The inequality $\phi(r)<2r$ implies that $x_r$ lies outside $P$ and we observe that because of the distance condition, $x_r\neq x_r'$ for distinct $r$ and $r'$ in $\R'$. Note that $\phi(r_i)>\phi(r_j)$ implies that $x_{r_j}\in T_{\phi(r_i)}$.

Denote by $r_1<r_2<\cdots<r_N$ the elements of $\R'$, and let $m$ the smallest index such that $\phi(r_m)\geq \phi(r_N)$. By the preceding discussion, $T_{\phi(r_m)}$ contains all vertices $x_{r_i}$ for every $i<m$ and contains also at least $r_N+1$ vertices from $P$. Hence, it has order at least $$m+r_N \geq m+r_m+(N-m)=r_m + N \geq r_m + |\R|/2$$ since the $r_i$ are distinct integers, which yields the result.
\end{proof}

\begin{comment}
We now may prove one of our main results giving the best-known upper bound on the burning number of connected graph.

\begin{proof}[Proof of Theorem~\ref{th:approx}]
Let $T$ be a tree on $n$ vertices and let $B=[0;p]$. By repeatedly applying Lemma~\ref{lem:approx_ind}, the tree $T$ is $B$-burnable unless $n > \sum_{i=0}^{p}(3i/2)=3p(p+1)/4$, which is not possible if $p=\ceil*{\sqrt{\frac43 n} }$, which is exactly the statement of the theorem.
\end{proof}
\end{comment}

\subsection{Graphs of minimum degree $3$ and $4$}

Using Theorem~\ref{th:approx}, we are now ready to argue that the BNC almost holds for graphs of minimum degree at least $3$, and obtain Theorems~\ref{th:degree3} and~\ref{th:degree4}. We use the following two convenient theorems.

\begin{theorem}\cite{kleitman1991spanning}\label{th:3leaves}
Every connected graph on $n$ vertices each of degree at least $3$ admits a spanning tree with at least $\frac{n}4+1$ leaves.
\end{theorem}

\begin{theorem}\cite{griggs1992spanning}\label{th:4leaves}
Every connected graph on $n$ vertices each of degree at least $4$ admits a spanning tree with at least $\frac{2n+8}5$ leaves.
\end{theorem}

\begin{proof}[Proof of Theorem~\ref{th:degree3}]
Let $G$ be a connected graph on $n$ vertices each of degree at least $3$. By Theorem~\ref{th:3leaves}, there is a spanning tree $T$ with at least $\frac{n}4+1$ leaves. Let $T'$ be the tree obtained from $T$ by deleting all leaves. Note that $b(T)\leq b(T')+1$. By Theorem~\ref{th:approx}, we have that $$b(T')\leq \ceil*{\sqrt{\frac43\cdot\left(\frac34n-1\right)}}+1.$$ We derive that $b(G)\leq b(T)\leq \lceil \sqrt{n} \rceil +2$.
\end{proof}

\begin{proof}[Proof of Theorem~\ref{th:degree4}]
Let $G$ be a connected graph on $n$ vertices each of degree at least $4$. By Theorem~\ref{th:4leaves}, there is a spanning tree $T$ with at least $\frac{2n+8}5$ leaves. Let $T'$ be the tree obtained from $T$ by deleting all leaves. Note that $b(T)\leq b(T')+1$. By Theorem~\ref{th:approx}, we have that $$b(T')\leq \ceil*{\sqrt{\frac43\cdot\left(\frac{3n-8}5\right)}}+1.$$ We derive that $b(G)\leq b(T)\leq \lceil \sqrt{n} \rceil$ when $n$ is large enough (for example, if $n\ge 325$).
\end{proof}

\section{Caterpillars Result}\label{sec:cat}
This section is devoted to the proof of Theorem \ref{th:reduced}. We start with another inductive result similar to Lemma \ref{lem:approx_ind}, that will be of use in the final step of the proof.

\subsection{Inductive Lemmas}

\begin{lemma}
    \label{lem:approx_cat}
    For every $k$-caterpillar $T$ and any integer $r$, either $\rad(T)\leq r$ or there exists a subtree $T'$ of $T$ such that 
    \begin{enumerate}
    \item $T\setminus T'$ is connected,
    \item $\rad(T')\leq r$, and
    \item $|V(T')|\geq 2r+1-k$.
    %\item %$r amongst the $|\R|/2$ largest elements of $\R$ (not useful for the theorem)
    \end{enumerate}
\end{lemma}

\begin{proof}
The proof is based on an analogous argument as the one of Lemma \ref{lem:approx_ind}: we consider a maximum length path $P=(t_0,t_1,\ldots,t_\ell)$ of $T$, root the tree $T$ in $t_\ell$, and define $\phi(r)$ to be the maximum $i$ such that the subtree rooted in $t_i$ is contained in $B(t_r,r)$. Either $\phi(r)=2r+1$, or there exists a vertex $v\not\in P$ such that 
\begin{enumerate}
\item The closest vertex from $v$ on $P$ is $t_{\phi(r)+1}$, and
\item $d(t_{\phi(r)+1},v)=2r-\phi(r)$.
\end{enumerate}
Since $T$ is a $k$-caterpillar and $P$ a maximum path, every vertex is at distance at most $k$ from $P$, and so $2r-\phi(r)\leq k$. The latter inequality implies the desired result since the tree rooted at $t_{\phi(r)}$ contains at least the vertices $t_0,t_1,\ldots,t_{\phi(r)}$.
\end{proof}

By applying repeatedly the previous lemma radius by radius we obtain the following result for $k$-caterpillars.

\begin{corollary}\label{coro:easycat}
Let $T$ be a tree of growth $k$ and of order $n$. If $\R$ is a set of integers such that $\sum_{r \in \R} (2r+1-k)\geq n$, then $T$ is $\R$-burnable.
\end{corollary}

Before starting the proof of the theorem, we state a final inductive lemma.

\begin{lemma}\label{lem:magicalSlemma}
Let $\R$ be a finite set of integers and denote by $R$ its maximum element. Let $T$ be a tree and $P$ a path in $T$ with two endpoints $u$ and $v$. Denote by $T_u$ (respectively, $T_v$) the connected component of $T \setminus (P \setminus \{u,v\})$ containing $u$ (respectively, $v$), and define $k=\max_{x \not\in V(T_u\cup T_v)}d(P,x)$.  If $T_u \cup T_v + uv$ is $(\R \setminus \{R\})$-burnable and $d(u,v)\leq 2R - 2 k+2$, then $T$ is $\R$-burnable.
\end{lemma}

\begin{proof}
Define $\R'=\R\setminus \{R\}$, $T' = (T_u \cup T_v + uv)$ and $T_P = T \setminus (T_u \cup T_v)$. By hypothesis, $T'$ is $\R'$-burnable so let us consider a set of vertices $(x_r)_{r\in \R'}$ such that the balls $B_{T'}(x_r,r)_{r \in \R'}$ cover $T'$. We want to keep in $T$ the same balls centered in the same vertices and add a ball of radius $R$ that covers every vertex not covered by those balls. Since $T$ is a tree, it is sufficient to show that no two non-covered vertices are at distance more than $2R$ (centering the ball at the middle of a longest path in $T$ between non-covered vertices covers all of them). 

Define $W=V(T)\setminus  \cup_{r\in \R'} B_T(x_r,r)$ and assume for a contradiction that there exists $(x,y)\in W^2$ such that $d_T(x,y)>2R$. %We distinguish several cases depending on if $x$, $y$ belong to either $T_u$, $T_v$ or $T_P$. 
We first define $\R_u = \{r\in \R' \colon x_r \in T_u\}$, similarly $\R_v=\R'\setminus \R_u$, and:
$$r_v=\max_{r\in \R_v}(r-d_T(x_r,v))\quad\text{ and }\quad r_u=\max_{r\in \R_u}(r-d_T(x_r,u)).$$
It is evident that $\max(r_u,r_v)<R$ and observe also that any vertex $w$ in $W$ satisfies $d(w,u)>r_u$ and $d(w,v)>r_v$. Therefore, if $z\in T_u \cap W$, we derive that $r_u<d(z,u)<r_v$. Hence, $r_u<r_v$ and so $T_u\cap W$ or $T_v\cap W$ is empty. 

Without loss of generality, assume that $r_u\leq r_v$, so that $T_v\cap W=\varnothing$. In particular, neither $x$ nor $y$ can belong to $T_v$. Moreover, for any $w\in W\cap T_P$, observe that $$d(w,u)+r_v<d(w,u)+d(w,v) = d(u,v)+2d(w,P) \leq 2R +2.$$

We are now able to derive a contradiction in every possible case:
\begin{enumerate}
\item If $(x,y)\in T_P^2 $, then $d(x,y)\leq d(x,P)+d(y,P)+d(u,v)-2\leq 2R$.
\item If $(x,y)\in T_u^2$, then $d(x,y)\leq d(x,u)+d(y,u) \leq 2r_v < 2R$.
\item If $(x,y)\in T_u \times T_P $, then  $d(x,y)=d(x,u)+d(y,u)< r_v+d(y,u) < 2R+2$.
\end{enumerate}
The proof follows.
\end{proof}

%------------------------------------------------------------------------------------------------------------------------------%
\subsection{Proof of Theorem \ref{th:maincat}}

Let $k$ and $c$ be integers, and define $f(k)=\fk$. If Theorem~\ref{th:maincat} is false, then there exists a $k$-caterpillar $T$ such that $b(T)>\lceil \sqrt{|V(T)|+c} \rceil > f(k)$, so let us assume now that $T$ is a minimum order such $k$-caterpillar.

%
%
%
%
%Let $T$ be a $k$-caterpillar such that $b(T)>\lceil \sqrt{|V(T)|+c} \rceil$, and assume $T$ is minimal for that property.  Notice that this is equivalent as saying that $T$ is a minimal tree such that there exists $R$ with $\sum_{r=0}^{R}(2r+1)\geq |V(T)|+c$ and $T$ cannot be covered by balls of radii $0,1,\ldots,R$. We want to prove that $|V(T)|+c<f(k)^2$.

To simplify notation, we will in the following denote $n=|V(T)|$, $R=\lceil \sqrt{n+c} \rceil-1$  and $\R$ the set of integers $\{0,1,\ldots,R\}$, so that $T$ is not $\R$-burnable and $R\geq f(k)$.\\

Here is a simple claim on $T$ that we will use often.

\begin{claim}\label{claim:min}
There is no integer $R'<R$ and subtree $T'$ of $T$ such that $T\setminus T'$ is covered by balls with radii between $R'$ and $R$ such that $|V(T)\setminus V(T')|\geq \sum_{r=R'}^{R}(2r+1)$.
\end{claim}
\begin{proof}
Since $|V(T)|+c\leq \sum_{r=0}^{R}(2r+1)$, we have that $$|V(T')|+c\leq \sum_{r=0}^{R'-1}(2r+1).$$ Therefore, $T'$ would be a smaller order counterexample contradicting the minimality of $T$.
\end{proof}

We now consider a \emph{spine} $S$ of $T$; that is, a path of maximum length, so that every vertex not in the path is at distance at most $k$ from the path. We arbitrarily choose $s_0$ to be one endpoint of $S$, and for two vertices $s'$ and $s''$ of $S$, we write $s'<s''$ if $d(s_0,s')<d(s_0,s'').$ In particular, $s'$ and $s''$ are distinct.
%We denote by $s_0,s_1,\ldots s_p$ the vertices of the spine, enumerated so that $s_is_{i+1}$ is an edge. 

Let us denote by $\R_{1}=\{r \in \R,\, (R+k)/2<r\leq R\}$. We now define the main object in our approach. A sequence of pairs $(c_i,r_i)\in V(S)\times \R_1$ for $1\leq i \leq t$ is \emph{admissible} if
\begin{enumerate}
\item For all $i<j$, $r_i \neq r_j$,
\item For all $i$, $c_i<c_{i+1}$ and  $d(c_i,c_{i+1})=r_i+r_{i+1}+1$,
\item For all $i$, $V_i=V\setminus (\cup_{j\leq i} B(c_j,r_j))$ induces a connected subtree $T_i$ of $T$, and
\item For all $i$, $S\cap V_i$ is a spine of $T_i$.
\end{enumerate}

Phrased differently, we pack balls along the spine without disconnecting the tree nor changing the spine.
%We now consider such an admissible collection $\C$ of pairs $(c_i,r_i)$ for  $1\leq i \leq t$ such that
%\begin{enumerate}R
%\item it maximizes $|\cup_{i} B(c_i,r_i)| $ amongst all admissible collection of pairs. \label{one}
%\item amongst all collections that satisfies \ref{one}, the sequence $(r_1,r_2,\ldots,r_t)$ is lexicographically maximal. \label{two}
%\end{enumerate}

\begin{claim}\label{claim:final} Let $\C$ be an admissible sequence $(c_i,r_i)$ for  $1\leq i \leq t$ such that the sequence of radii $(r_1,r_2,\ldots,r_t)$ is lexicographically maximal. We then have that
\begin{enumerate}
\item $t> |\R_1| -2k$, and
\item $|\!\cup_{i=1}^{t}  B(c_i,r_i)| \geq \sum_{i=1}^{t} (2r_i+1) + 2k^2$.
\end{enumerate}
\end{claim}
Note that the maximality condition implies that this maximal sequence may obtained by a process where at each step, we chose the largest remaining radii that keeps the sequence admissible.\\

Let us prove first that this claim implies our result.  Let $T'$ the subtree induced by $V\setminus \cup_{i=1}^{t}B(c_i,r_i)$. Since $\R'=\R_1\setminus \{r_i \in \R_1, \ 1\leq i \leq t\}$ contains less than $2k$ elements by the first item of the claim, we can apply $2k$ times Lemma \ref{lem:approx_cat}  to obtain a subtree $T''$ such that $T'\setminus T''$ is covered by balls with radius in $\R'$ and $|V(T')\setminus V(T'')|\geq \sum_{r\in \R'} (2r+1)-2k^2$. Overall, $T\setminus T''$ is coverable by balls of radius in $\R_1$ and $|V(T)\setminus V(T'')|\geq \sum_{r\in \R_1} (2r+1)$, which is impossible by Claim \ref{claim:min}.

\begin{proof}[Proof of Claim \ref{claim:final}] Consider $\C$ as in the statement and fix some integer $j\leq t$. Denote $V'=V\setminus \cup_{i\leq j} B(c_i,r_i)$. Note that $T'=T[V']$ is a $k$-caterpillar, and that by assumption $S'=S\cap V'$ is a spine of $T'$. Let $s'_0<s'_1<\ldots < s'_p$ the vertices of $S'$.  For a vertex $v$ not in the spine, we will say that $v$ is \emph{pending to} $s'_i$ if $s'_i$ is the closest vertex to $v$ on the spine. Let $\R'=\R_1\setminus \{r_i \in \R_1, \ 1\leq i \leq j\}$. Now by maximality of $\C$ we have that if $r\in \R'$, then:

% By maximality of $\C$, for any $r\in \R'=\R\setminus \{r_i, \ 1\leq i \leq j\}$, with $r>r_j$, the pair $(s'_r,r)$ cannot be added to $\C$, so it means that 
\begin{enumerate}
\item Either $j<t$ and $r\leq r_{j+1}$, or
\item The removal of $B(s'_r,r)$ disconnects $T'$, or
\item The vertices $s'_{2r+1},s'_{2r+2},\ldots,s'_p$ are no longer a spine of $T'\setminus B(s_r,r)$.
\end{enumerate}

If $j=t$ or $r> r_{j+1}$, then this implies that there is a vertex $v_r \not\in S'$ such that:
\begin{enumerate}
\item Either $v_r$ is in the subtree pending to $s'_i$ for some $ 2r-k\leq i \leq 2r$ and $d(s'_0,v_r)=2r+1$, as $v_r$ is not covering it, or
\item The vertex $v_r$ is in the subtree pending to $s'_i$ for some $2r< i \leq 2r+k$ and $d(s'_{0},v_r)=2i-2r$, as the distance from $s'_i$ to that node must be 1 plus the distance to the last node covered by $s'_{2r+1}$.
\end{enumerate}

We want to prove that all such $v_r$ are distinct. Let us say that $v_r$ is \emph{type }$1$ of it corresponds to the first case above, and of \emph{type} $2$, otherwise. See Figure \ref{fig:types} for representation of the two possibilities. 

The distance to $s'_0$ condition implies that elements of type $1$ are pairwise distinct, and same for elements of type $2$. We conclude that the parity of this distance is odd in one case and even in the other, which also forbids an element of type $1$ to be equal to an element of type $2$.

\begin{figure}[htbp]
\begin{center}
\includegraphics[scale=1.5]{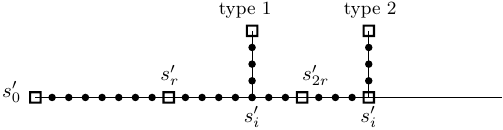}
\caption{Vertices of type 1 and 2.}\label{fig:types}
\end{center}
\end{figure}

If $j=t$, then this implies that we have at least $|\R'|$ vertices outside the spine such that each of them are pending to some $s'_i$ for $i$ between $\min\{2r-k, r\in \R'\}\geq 3k$ (remember all elements in $\R_1$ are at least $(R+k)/2\geq 2k$) and $\max\{2r+k, r\in \R'\}\leq 2R+k$. We now use Lemma \ref{lem:magicalSlemma} applied to $u=s'_{3k-1}$ and $v=s'_{2R+k+1}$. We have that $d(u,v)=2R-2k+2$ and since $T$ is not $\R$-burnable, the tree $T'$ denoted $T_u+T_v+uv$ in the statement of the lemma should not be $\R\setminus \{R\}$-burnable. By the minimality of $T$, this implies that $T\setminus T'$ contains at most $2R$ elements, which in turn implies that $\R'$ contains less than $2k$ elements, which is exactly the first item of the claim.

If $j<t-1$, then we denote by $\R''$ the set $\{r\in \R', r>r_{j+1}\}$. When $\R''$ is not empty, we denote by $r''$ the largest element of $\R''$. The discussion above implies that there are at least $|\R''|$ vertices outside the spine that are pending between $s'_{2r_{j+1}-k}$ and $s'_{2r''+1+k}$. By the definition of $\R_{1}=\{r \in \R,\, (R+k)/2<r\leq R\}$, we have that $$2r''+1+k\leq 2R+k+1< 2r_{j+1}+2r_{j+2}+1-k,$$ which implies that all those vertices outside the spine will be included into either $B(c_{j+1},r_{j+1})$ or $B(c_{j+2},r_{j+2})$. In fact, they cannot be included in $B(c_{j+1},r_{j+1})$ as for any $r \in R''$, $d(s'_0,v_{r}) = 2r+1 > 2r_{j+1}+1$, therefore they are only included in $B(c_{j+2},r_{j+2})$ and are not counted twice.

The total number of vertices outside the spine that will be included into the union of the balls in our sequence is at least the number of times balls in $\R_1$ are rejected for a smaller ball in the process. We claim that this number is at least $R-\min_{i\leq t}{r_i}$. 

We denote $m_j= \min_{i\leq j} r_i$  and observe first that at any time of the process, $\{r_i, 1\leq i\leq j\}$ is never equal to the interval $[m_j,R]$ by Claim \ref{claim:min}. Hence, there is at least one radius of size greater that $m_j$ that has been rejected. Now consider each step $j$ where we pick a new minimum, that is $m_j=r_j$. We have rejected at least $m_{j-1}-m_j$ radii (at least $1$ greater than $m_{j-1}$ and $m_{j-1}-m_j-1$ between $m_j$ and $m_{j-1}$). Overall, this gives that the number of total rejections is at least $R-m_t$ as claimed and this concludes the proof since $$R-m_t\geq t \geq  |\R_1|-2k +1 \geq (R-k)/2 - 2k +1 \geq 2k^2,$$ where the last inequality comes from the assumption $R\geq f(k)$. \end{proof}

\section{Conclusion and Further Directions}
%\todo{ A améliorer }
We provided the best known upper bound for the burning number of a connected graphs, and have shown that the burning graph conjecture is asymptotically true for $k$-caterpillars. We can in fact prove that this extends to Conjecture~\ref{bnc:strong} with analogous arguments, though we do not go into details. In this case, we can also do a finer analysis to modify the bound in Theorem~\ref{th:maincat} from $O(k^4)$ to $O(k^2)$, though it did not seem to us enough to justify an additional few pages of proofs.

The two proofs of Theorem~\ref{th:approx} and Theorem~\ref{th:maincat} may seem similar in approach: in both cases, we consider a maximum order path and try to place large balls centered on the path in a greedy fashion. However, the induction makes a big difference: in the case of Theorem~\ref{th:maincat}, we pick a maximum path initially and all balls chosen will be centered on the path, whereas in the proof the Theorem~\ref{th:approx}, at each step we take a new maximum path before removing a ball centered on it, so that in the next step, the new maximum path might not be a subpath of the previous one. In this proof, the set of centers constructed may not lie on a common path. 
One might wonder if the BNC could be true with the additional requirement that all balls are centered along some path. This is not the case, as can be seen by considering the graph obtained from a star with p leaves by subdividing each edge p times.

Theorem \ref{th:reduced} states that, for a fixed value of $k$, there are a finite number of graphs that could be counterexample of the BNC. As it is now, the statement induces a large number of cases, it would be interesting to reduce this number enough to obtain a computer-based proof of the BNC for small values of $k$.

%\todo{more discussion?}

\bibliographystyle{alpha}
\bibliography{biblio.bib}

\end{document}